\documentclass{elsarticle}

\usepackage[all,cmtip]{xy}
\usepackage{amsmath, amssymb}
\usepackage{graphics}
\usepackage{caption}
\usepackage{subcaption}
\usepackage{amsthm}
\usepackage{algorithm}
\usepackage{algorithmic}
\usepackage{booktabs}
\usepackage{color}

\def\deg{{\rm deg}}

\def\Bbb#1{\mathbb{#1}}

\def\bft{\mathbf{t}}

\def\Tr{\mathrm{Tr}}

\def\bfb{\boldsymbol{b}}
\def\bfc{\boldsymbol{c}}
\def\bfx{\boldsymbol{x}}

\def\RR{\mathbb{R}}

\def\OOO{\mathcal{O}}
\def\tOOO{\tilde{\mathcal{O}}}

\newtheorem{theorem}{{\bf Theorem}}

\newtheorem{corollary}[theorem]{{\bf Corollary}}
\newtheorem{proposition}[theorem]{{\bf Proposition}}
\newtheorem{lemma}[theorem]{{\bf Lemma}}

\pdfoutput=1

\begin{document}

\begin{frontmatter}


\title{Involutions of polynomially parametrized surfaces}

\author[a]{Juan Gerardo Alc\'azar\fnref{proy}}
\ead{juange.alcazar@uah.es}
\author[a]{Carlos Hermoso}
\ead{carlos.hermoso@uah.es}

\address[a]{Departamento de F\'{\i}sica y Matem\'aticas, Universidad de Alcal\'a, E-28871 Madrid, Spain}

\fntext[proy]{Member of the Research Group {\sc asynacs} (Ref. {\sc ccee2011/r34})}

\begin{abstract}
We provide an algorithm for detecting the involutions leaving a surface defined by a polynomial parametrization invariant. As a consequence, the symmetry axes, symmetry planes and symmetry center of the surface, if any, can be determined directly from the parametrization, without computing or making use of the implicit representation. The algorithm is based on the fact, proven in the paper, that any involution of the surface comes from an involution of the parameter space $\RR^2$; therefore, by determining the latter, the former can be found. The algorithm has been implemented in the computer algebra system Maple 17. Evidence of its efficiency for moderate degrees, examples and a complexity analysis are also given. 
\end{abstract}
\end{frontmatter}

\section{Introduction}

Symmetry detection in 3D objects is an important matter in fields like Computer Graphics or Computer Vision. In Computer Graphics, it is useful in order to gain understanding when analyzing pictures, and also to perform tasks like compression, shape editing or shape completion. In Computer Vision, symmetry is important for object detection and recognition. Many techniques have been tried so far to solve the problem. Some of them involve statistical methods and, in particular, clustering; see for example the papers \cite{Berner08, Bokeloh, Mitra06, Podolak}, where the technique of transformation voting is used, or \cite{Sun}, based on the Extended Gauss Image. Other techniques are robust auto-alignment \cite{Simari}, spherical harmonic analysis \cite{Martinet}, feature points \cite{Loy}, primitive fitting \cite{Schnabel}, and spectral analysis \cite{Lipman}, to quote just a few. In addition, there are algorithms for computing the symmetries of 2D and 3D discrete objects \cite{Alt88, Brass, Jiang, Li08} and for boundary-representation models \cite{Li08, Li10, Tate}. The list of all papers addressing the subject is really very long, and the interested reader is referred to the bibliographies in these papers to get a more complete list. 

In the references on the topic, the object to be analyzed is quite commonly a point cloud or a mesh, sometimes with missing parts, so that little structure is assumed on it. One exception here is the case of tensor product surfaces \cite{Alt88}. In this case the geometry of the object, and in particular its symmetry, follows from that of a discrete object behind, the {{\it control polyhedron}. Hence, the symmetries of the object can be found by applying methods to detect symmetries of discrete objects \cite{Alt88, Brass, Jiang, Li08}.

In this paper we consider the problem of computing \emph{involutions}, i.e. symmetries with respect to a point, line or plane, of objects with a stronger structure, namely the set $S$ of points defined by a polynomial parametrization 
\[\bfx(t,s)=(x(t,s),y(t,s),z(t,s)),\]with $(t,s)\in {\Bbb R}^2$. Such objects, well-known in Constructive Algebraic Geometry and Computer Aided Geometric Design, are called \emph{polynomially parametrized} algebraic surfaces. Certainly a tensor product surface corresponds to this description whenever $(t,s)$ is restricted to a compact rectangle $[a,b]\times [c,d]\subset {\Bbb R}^2$. However, in our case $(t,s)$ takes values over the whole plane ${\Bbb R}^2$. Therefore we deal with the global surface $S$, not just with a piece of it, and an approach like \cite{Alt88, Brass, Jiang, Li08} is not applicable here. 

In order to solve the problem we assume {\it good} properties on the parametrization $\bfx(t,s)$. More precisely, we assume that the parametrization is injective except perhaps for a closed subset of (possibly singular) points of $S$, and that it is also surjective as a mapping from the plane to $S$. Under these conditions, we prove that any involution of the surface is the result of lifting an involution of the plane to $S$ via the parametrization of the surface. This way, the problem is translated to the parameter space, and in turn shown to be solvable by applying bivariate real polynomial system solving. 

The method can be seen as the generalization to surfaces of some ideas recently applied to compute symmetries of planar and space rational curves \cite{Alcazar13, AHM13,  Alcazar.Hermoso13, AHM15}. Furthermore, the problem treated here is related to the more general question of extracting geometric invariants from a surface parametrization. This question appears as one of the eight open problems on the interplay between Algebraic Geometry and Geometric Modeling posed by Prof. Ron Goldman in \cite{goldman}. 

The structure of the paper is the following. In Section \ref{gen-surf} we introduce some generalities on surface parametrizations and isometries, and we prove several results on symmetries of surfaces; although rotational symmetry is not addressed in the paper, some properties of this type of symmetry are considered here, and then used to prove certain facts on involutions. The method itself is presented in Section \ref{first-sec}. Section \ref{sec-cylindrical} briefly addresses the special case of cylindrical surfaces. Finally, in Section \ref{sec-perf} we provide two detailed examples, we address complexity issues, and we report on the practical implementation of the algorithm carried out in the computer algebra system Maple 17. Our conclusions, and some observations about future work, are provided in Section \ref{conclu}.

\vspace{0.3 cm}
\noindent {\bf Acknowledgements.} We want to thank the reviewers of the paper for his/her valuable comments, which helped to improve the quality of the paper.

\section{Generalities} \label{gen-surf}

\subsection{Properness and normality.}\label{prop-norm}

Throughout the paper we consider an algebraic surface $S \subset \RR^3$ different from a plane, polynomially parametrized by $\bfx: \RR^2 \rightarrow \RR^3$, where \[\bfx(t,s)=(x(t,s),y(t,s),z(t,s))\]and $x(t,s),y(t,s),z(t,s)$ are polynomials in the variables $t,s$ with coefficients in ${\Bbb Q}$. Nevertheless, at certain points of the paper we will implicitly assume that the parametrization $\bfx$ can be also considered as $\bfx:{\Bbb C}^2\to {\Bbb C}^3$, so that both the parameter space and the surface can be embedded into the complex plane and the complex space. Similarly for other real mappings in the paper. Since $S$ admits a rational, and in fact polynomial, parametrization then in particular $S$ is irreducible. The functions $x(t,s),y(t,s),z(t,s)$ are the {\it components} of $\bfx$, while $t,s$ are the {\it parameters} of $\bfx$. We define the {\it total degree}, $n$, of the parametrization $\bfx$ as the maximum of the total degrees of the components of $\bfx$. Furthermore, we will assume that ${\bfx}$ is {\it proper}, i.e. birational or equivalently injective for almost all points of $S$ except for at most a closed subset of $S$. In particular, this implies that ${\bfx}^{-1}$ is a rational map. One can check properness by using the algorithms in \cite{PDSS02, PDS04}; for reparametrization questions one can see \cite{ARTSV11, Li, PD06, PD13}. 

We say that the parametrization $\bfx(t,s)$ is \emph{normal} if it is surjective, i.e. if every point of $S$ is reached via $\bfx$ by some pair of parameters $(t,s)\in {\Bbb C}^2$. This problem has been well studied for the case of rational curves \cite{SWPD}. However, the same problem for surfaces is not completely well understood yet. The question has been addressed in \cite{Bajaj, Chou} for special kinds of surfaces, and also in \cite{PSV}, where partial results on the problem are presented. In particular, in \cite{PSV} a sufficient condition for a polynomial parametrization to be normal (see Corollary 3.15 and Corollary 4.4 therein) is given. Throughout the paper, we will also assume that the parametrization $\bfx$ we work with is normal. 

Additionally, for technical reasons we will require $\bfx(0)$ to be a regular point of $S$; this requirement can always be satisfied by applying a random linear change of parameters, if necessary. 

\subsection{Isometries of algebraic surfaces.} \label{isom}

Let us recall some facts from Euclidean geometry \cite{Coxeter}. An \emph{isometry} of $\RR^n$ is a map $f:\RR^n\longrightarrow \RR^n$ preserving Euclidean distances. Any isometry $f$ of $\RR^n$ is linear affine, taking the form \begin{equation}\label{eq:isometry}
f(x) = Qx + \bfb, \qquad x\in \RR^n,
\end{equation}
with $\bfb \in \RR^n$ and $Q\in \RR^{n\times n}$ an orthogonal matrix. In particular $\det(Q) = \pm 1$. For $n=3$, the isometries of the space form a group under composition that is generated by reflections, i.e., symmetries with respect to a plane, also called \emph{mirror symmetries}. An isometry is called \emph{direct} when it preserves the orientation, and \emph{opposite} when it does not. In the former case $\det(Q) = 1$, while in the latter case $\det(Q) = -1$. The identity map $\textup{id}_{\RR^n}$ of $\RR^n$ is called the \emph{trivial symmetry}. An isometry $f(x) = Qx + \bfb$ of $\RR^n$ is called an \emph{involution} if $f\circ f = \textup{id}_{\RR^n}$, in which case $Q^2 = I$ is the identity matrix and $\bfb\in\ker(Q+I)$.

 In the case $n=3$, the classification of the nontrivial isometries of Euclidean space includes reflections (in a plane), rotations (about an axis), and translations, which combine in commutative pairs to form twists (compositions of a rotation about an axis and a translation parallel to the axis of rotation), glide reflections (composition of a reflection in a plane and a translation parallel to the reflection plane), and rotatory reflections (composition of a rotation and a reflection). Composing three reflections in mutually perpendicular planes through a point $P$, yields a \emph{central inversion}, also called a \emph{central symmetry} with center $P$, i.e., a symmetry with respect to the point $P$, called the {\it symmetry center}. 

Additionally, we say that $S$ has \emph{rotational symmetry} if there exist a line $\ell$ and a real $\theta$ such that $S$ is invariant under the rotation ${\mathcal R}_{\ell,\theta}$ about $\ell$, by an angle $\theta$. Furthermore, in that case we say that $\ell$ is an \emph{axis of rotation} of $S$. The special case of rotation by an angle $\pi$ is of special interest, and is called a \emph{half-turn}, or an \emph{axial symmetry}; in that case, the axis of rotation is called the {\it symmetry axis}. Central inversions, reflections and axial symmetries are involutions. Furthermore, axial symmetries are direct, while reflections and central inversions are opposite.

In the rest of the section, we will examine some questions related to algebraic surfaces that are invariant under isometries. For this purpose, we will consider a real algebraic surface $S$, not a plane. We will say that $S$ is a {\it cylindrical surface} (or a {\it cylinder}) if $S$ is a ruled surface whose generatrices are all of them parallel to a given direction. We will say that $S$ is a {\it surface of revolution} if there exists a line $\ell$, called the {\it axis of revolution} of $S$, such that $S$ is invariant under \emph{any} rotation about $\ell$. Notice that an axis of revolution of $S$ is an axis of rotation of $S$ for {\it any} angle $\theta$; however, an axis of rotation $\ell$ of $S$ is not necessarily an axis of revolution of $S$, since in general there are only certain values of $\theta$ such that $S$ is invariant under the rotation ${\mathcal R}_{\ell,\theta}$.

\begin{proposition} \label{th-gen}
If $S$ is invariant under any of the following transformations:
\begin{itemize}
\item [(i)] a translation by a nonzero vector $\bar{a}\in \RR^3$, 
\item [(ii)] a twist operation about an axis $\ell$, by a nonzero vector $\bar{a}$ (parallel to $\ell$),
\item [(iii)] a glide reflection about a plane $\pi$, by a nonzero vector $\bar{a}$ (parallel to $\pi$),
\end{itemize}
then $S$ is a cylindrical surface and its generatrices are parallel to $\bar{a}$. 
\end{proposition}

\begin{proof} As for (i) or (iii), we observe that for every $P\in S$, the line ${\mathcal L}_{P,\bar{a}}=P+\lambda \cdot \bar{a}$ has infinitely many points in common with $S$. Since $S$ and ${\mathcal L}_{P,\bar{a}}$ are algebraic, ${\mathcal L}_{P,\bar{a}}$ must be contained in $S$, and the result follows. So let us address (ii). Suppose that $S$ is invariant under a twist operation ${\mathcal T}$ about an axis $\ell$ by an angle $\theta$, by a translation vector $\bar{a}$ (parallel to $\ell$). If $S$ is a union of circular cylinders of axis $\ell$ then $S$ is cylindrical, with generatrices parallel to $\bar{a}$, and we have finished. Otherwise, there is a circular cylinder of axis $\ell$ such that its intersection with $S$ yields a real algebraic space curve ${\mathcal C}$. This curve ${\mathcal C}$ is also invariant under ${\mathcal T}$. Therefore, the Zariski closure ${\mathcal C}^{\star}$ of the projection of ${\mathcal C}$ onto a plane $\Pi$ normal to $\ell$ is invariant under a rotation by an angle $\theta$, about the point $\Pi \cap \ell$. 

If ${\mathcal C}^{\star}$ is not a circle, then by Lemma 2 in \cite{LR08} we must have $\theta=2\pi/n$, where $n\in {\Bbb N}$. Hence, given any point $P\in S$, we have that $\{P,{\mathcal T}^n(P),{\mathcal T}^{2n}(P),\ldots\}$ is a sequence of points of the line ${\mathcal L}_{P,\bar{a}}=P+\lambda\cdot \bar{a}$. All the elements of this sequence lie in $S$. Hence, ${\mathcal L}_{P,\bar{a}}$ intersects $S$ at infinitely many points, and since ${\mathcal L}_{P,\bar{a}}$, $S$ are algebraic, therefore  ${\mathcal L}_{P,\bar{a}}\subset S$. Since this happens for every point $P\in S$, we deduce that $S$ is cylindrical; furthermore, the generatrices of $S$ are parallel to $\bar{a}$.

So suppose that ${\mathcal C}^{\star}$ is a circle, in which case Lemma 2 in \cite{LR08} does not apply. If $\theta=2\pi/n$, $n\in {\Bbb N}$, we can use the above argument. So let us address the case $\theta\neq 2\pi/n$. Without loss of generality, we can assume that $\ell$ is the $z$-axis and $\Pi$ is the $xy$-plane, so ${\mathcal C}^{\star}$ is $\{x^2+y^2=d^2,z=0\}$, with $d\neq 0$. Since ${\mathcal C}$ is invariant under ${\mathcal T}$ it is non-bounded in $z$; however, ${\mathcal C}$ is contained in the cylinder $x^2+y^2=d^2$ and therefore it is certainly bounded in $x,y$. Since ${\mathcal C}$ is real, non-bounded in $z$ and ${\mathcal C}$ is contained in $x^2+y^2=d^2$, which is a circular cylinder whose axis is the $z$-axis, the projective closure of ${\mathcal C}$ must contain the point at infinity corresponding to the $z$-axis, $p_{\infty}=(0:0:1:0)$. Now ${\mathcal C}$ is algebraic, and therefore it cannot wind infinitely around the $z$-axis; so $p_{\infty}$ corresponds to a real asymptote ${\mathcal L}$ of ${\mathcal C}$ parallel to the $z$-axis. However since $S$ is invariant under ${\mathcal T}$, for any $k\in {\Bbb N}$ we have that ${\mathcal T}^k({\mathcal L})$ is also an asymptote of ${\mathcal C}$. Since $\theta\neq 2\pi/n$ we get that 
${\mathcal T}^k({\mathcal L})\neq {\mathcal T}^p({\mathcal L})$ for $k\neq p$. Therefore ${\mathcal C}$ has infinitely many asymptotes. But this is impossible because ${\mathcal C}$ is algebraic.


\end{proof}

We will say that $S$ has {\it translational symmetry}, {\it twist symmetry} or {\it glide reflection-symmetry} if it is invariant with respect to a translation, a twist, or a glide reflection. From Proposition \ref{th-gen}, any surface showing some of these symmetries must be cylindrical. Thus, in the rest of the section we focus on non-cylindrical surfaces. We first state the following lemma, where we summarize several results on the composition of isometries which we will use later. We refer to Section 7.3 in \cite{Coxeter}, Chapter 1 of \cite{Heard} and Chapter 19 of \cite{Fernandez} for proofs. 

\begin{lemma} \label{composition}
\begin{itemize}
\item [(1)] The composition of two central inversions of centers $p_1\neq p_2$ is a translation by a vector parallel to $\overline{p_1p_2}$. 
\item [(2)] The composition of two rotations ${\mathcal R}_{\ell_1,\alpha}$, ${\mathcal R}_{\ell_2,\beta}$, with concurrent axes $\ell_1\neq \ell_2$ forming an angle $\Phi$, is another rotation ${\mathcal R}_{\ell_3,\gamma}$, where $\ell_3$ is concurrent with $\ell_1, \ell_2$, and 
\begin{equation}\label{cos}
\cos\left(\frac{\gamma}{2}\right)=\cos\left(\frac{\alpha}{2}\right)\cdot \cos\left(\frac{\beta}{2}\right)-\sin\left(\frac{\alpha}{2}\right)\cdot \sin\left(\frac{\beta}{2}\right)\cdot \cos(\Phi).
\end{equation}
\item [(3)] The composition of two rotations with skew axes is a twist.
\item [(4)] The composition of two reflections with respect to two planes $\pi_1\neq \pi_2$ is a translation when $\pi_1,\pi_2$ are parallel, and a rotation about the line $\pi_1\cap \pi_2$, of angle equal to double the angle formed by $\pi_1,\pi_2$, when $\pi_1,\pi_2$ are not parallel.  
\end{itemize}
\end{lemma}

\begin{proposition} \label{several}
If $S$ is not cylindrical, the symmetry center of $S$, if it exists, is unique.
\end{proposition}

\begin{proof} Assume that $S$ has two symmetry centers, i.e. that it is invariant under two different central inversions. By part (1) of Lemma \ref{composition}, $S$ is invariant under a translation. Therefore by Proposition \ref{th-gen} $S$ must be cylindrical, which cannot happen by hypothesis.
\end{proof}

Now we address some results on the rotational symmetries of $S$. We first need the following instrumental lemma, which concerns planar algebraic curves. We say that a planar algebraic curve ${\mathcal C}$ has \emph{rotational symmetry} if it is invariant under some planar rotation, i.e. a rotation about a point.  

\begin{lemma} \label{rot-planar}
Let ${\mathcal C}$ be a planar algebraic curve of degree $d$ with rotational symmetry about a point $P$, by an angle $\theta$. If ${\mathcal C}$ is not a product of circles centered at $P$, then $\theta=\frac{2\pi}{m}$ where $0<m\leq 2d$, $m\in {\Bbb N}$.
\end{lemma}

\begin{proof} If ${\mathcal C}$ is a line (with multiplicity $d$) then $m=2$ and the result holds. So assume that ${\mathcal C}$ is not a line. Since by hypothesis ${\mathcal C}$ is not a product of circles with the same center, there exists a circle $C^{\star}$ centered at $P$ which, by Bezout's Theorem, intersects ${\mathcal C}$ in at most $2d$ points, counted with multiplicity. On the other hand, by Lemma 2 in \cite{LR08} we must have $\theta=\frac{2\pi}{m}$, where $m\in {\Bbb N}$. Let ${\mathcal R}_{P,\frac{2\pi}{m}}$ be the rotation about $P$, by an angle $\frac{2\pi}{m}$, and let $Q\in {\mathcal C}\cap C^{\star}$. Notice that ${\mathcal C}\cap C^{\star}$ is invariant under ${\mathcal R}_{P,\frac{2\pi}{m}}$. Then the sequence $\{Q,{\mathcal R}_{P,\frac{2\pi}{m}}(Q),{\mathcal R}^2_{P,\frac{2\pi}{m}}(Q),\ldots\}$ consists of at most $m$ different points, and the union of all these points is ${\mathcal C}\cap C^{\star}$. Therefore, $m\leq 2d$.
\end{proof}

\begin{lemma} \label{angle}
Let $S$ be an algebraic surface of degree $d$, invariant under a rotation about an axis $\ell$, by a non-trivial angle $\theta$ (i.e. $\theta\neq 2k\pi$, with $k\in {\Bbb Z}$). If $\ell$ is not an axis of revolution of $S$, then $\theta=\frac{2\pi}{m}$, where $0<m\leq 2d$, $m\in {\Bbb N}$.
\end{lemma}

 \begin{proof} Since by hypothesis $\ell$ is not a revolution axis, there exists a plane $\Pi$, normal to $\ell$, such that the intersection curve ${\mathcal C}=\Pi \cap S$ is not a product of circles centered at $P=\Pi \cap \ell$. Now ${\mathcal C}$ is an algebraic planar curve of degree at most $d$. Furthermore, since $S$ has rotational symmetry about $\ell$ the curve ${\mathcal C}$ is invariant under the (planar) rotation about the point $P=\Pi\cap \ell$, by the angle $\theta$. But then the result follows from Lemma \ref{rot-planar}.
 \end{proof}

Lemma \ref{angle} provides the following corollary.

\begin{corollary} \label{finiteang}
If $S$ is not a surface of revolution, then the number of angles $\theta\in [0,2\pi)$ such that $S$ is invariant under some rotation ${\mathcal R}_{\ell,\theta}$ is finite. 
\end{corollary}

Notice that Corollary \ref{finiteang} does not imply that $S$ has finitely many axes of rotation, since there might be infinitely many $\ell$'s such that $S$ is invariant under a rotation about $\ell$. Nevertheless, our goal is to prove such statement. But in order to do this, we first need the following result.

\begin{lemma} \label{parallel-rot}
If $S$ is not a plane, then $S$ cannot have two different parallel axes of rotation.
\end{lemma}

\begin{proof} Let $\ell$ be an axis of rotation of $S$. For a given plane $\Pi$ normal to $\ell$, a rotation of $S$ about $\ell$ induces a rotation of the planar curve $S\cap \Pi$ around the point $\ell \cap \Pi$. Now if $S$ has another axis of rotation $\ell'\neq \ell$, parallel to $\ell$, then $S\cap \Pi$ exhibits rotational symmetry around two different centers of rotation (the intersections of $\Pi$ with $\ell,\ell'$). However, the center of rotation of an algebraic planar curve other than a line, if any, is unique (see Theorem 5.3 in \cite{LRTh}). So $S\cap \Pi$ must be a line. Since this must happen for any plane $\Pi$ normal to $\ell$, we deduce that $S$ is a plane, which is excluded.
\end{proof}

\begin{proposition} \label{finite-rot-axes}
Let $S$ be a non-cylindrical real algebraic surface. 
\begin{itemize}
\item [(1)] All the axes of rotation of $S$ intersect at a point. 
\item [(2)] If $S$ is not a surface of revolution then $S$ has finitely many axes of rotation.
\end{itemize}
\end{proposition}

\begin{proof} (1) Suppose that $S$ has several axes of rotation. By Lemma \ref{parallel-rot}, not two of them can be parallel. By part (3) of Lemma \ref{composition}, the composition of two rotations whose axes do not intersect is a twist. Hence if $S$ has two axes of rotation that do not intersect, then $S$ is invariant under two rotations with non-concurrent axes, and therefore it is invariant by a twist. Therefore, by Proposition \ref{th-gen} $S$ is cylindrical, which cannot happen by hypothesis.

(2) Suppose that $S$ has infinitely many axes of rotation, all of which, by the first part of the statement, share a point $P$. Therefore, these axes form infinitely many different angles with each other. Since $S$ has infinitely many axes of rotation, $S$ is invariant under infinitely many rotations ${\mathcal R}_{\ell,\theta}$. Let ${\mathcal R}_{\ell_1,\alpha}$, ${\mathcal R}_{\ell_2,\beta}$, where $\ell_1\neq \ell_2$, $\ell_1\cap \ell_2=P$, be any two of these rotations, and let $\Phi$ be the angle between $\ell_1,\ell_2$. By part (2) of Lemma \ref{composition}, the composition of ${\mathcal R}_{\ell_1,\alpha}$, ${\mathcal R}_{\ell_2,\beta}$ is another rotation ${\mathcal R}_{\ell_3,\gamma}$, where $\ell_3$ is concurrent with $\ell_1, \ell_2$ (i.e. $\ell_3$ also goes through $P$), and $\gamma$ is related with $\alpha,\beta,\Phi$ according to Equation \eqref{cos}. However, by Corollary \ref{finiteang}, $\alpha,\beta,\gamma$ can have just finitely many values. This yields finitely many values for $\Phi$ too. However this is a contradiction, because $S$ has infinitely many axes of rotation through $P$, which therefore form infinitely many different angles with each other.  
\end{proof}

Notice that if $S$ is a surface of revolution then we can certainly have infinitely many axes of rotation. For instance, if $S$ is an ellipsoid of revolution generated by rotating the ellipse \[\left\{\frac{y^2}{a^2}+\frac{z^2}{b^2}=1,x=0\right\}\]about the $z$-axis, any line contained in the plane $z=0$, through the origin, is a symmetry axis of $S$, and therefore an axis of rotation (the rotation angle is $\pi$, in that case). 

\begin{corollary} \label{finite-axial-sym}
Let $S$ be a non-cylindrical real algebraic surface. 
\begin{itemize}
\item [(1)] All the symmetry axes of $S$ intersect at a point. 
\item [(2)] If $S$ is not a surface of revolution, then $S$ has finitely many symmetry axes.
\end{itemize}
\end{corollary}

If $S$ is a surface of revolution then it can have infinitely many symmetry axes, as it happens in the case of the ellipsoid of revolution. However, not every surface of revolution has infinitely many symmetry axes. For instance, the paraboloid $x^2+y^2=z$ has just one symmetry axis, namely the $z$-axis.

\begin{proposition} \label{rotaxis}
Let $S$ be an irreducible algebraic surface of revolution, not a plane or a sphere. Then the axis of revolution of $S$ is unique.
\end{proposition}

\begin{proof} Assume to the contrary that $S$ has two different axes of revolution, $\ell_1$ and $\ell_2$, which cannot be parallel because of Lemma \ref{parallel-rot}. If $\ell_1$ and $\ell_2$ are skew lines, then by Lemma \ref{composition}, statement (3), $S$ has twist symmetry, which implies that $S$ is cylindrical; in that case, since $S$ is a surface of revolution it must be a circular cylinder, and the statement follows. If $S$ is not cylindrical, then $\ell_1$ and $\ell_2$ intersect at a point $P$. Since $\ell_2$ is an axis of revolution of $S$, whenever $S$ is not a plane we can find a plane $\pi$ normal to $\ell_2$ such that $S\cap \pi$ is a circle $C$. For every point of $C$, the distance to $P$ is the same; let this distance be $r$. Now since $\ell_1$ is also an axis of revolution of $S$, by rotating $C$ around $\ell_1$ we generate a 2-dimensional piece of a sphere $\mbox{Sph}_{P,r}$ centered at $P$, with radius equal to $r$. Since the intersection of $S$ and $\mbox{Sph}_{P,r}$ is a two-dimensional subset and since $S, \mbox{Sph}_{P,r}$ are irreducible and algebraic, then by Study's Lemma (see Section 6.13 of \cite{Fischer}) $\mbox{Sph}_{P,r}=S$, i.e. $S$ is a sphere. 
 \end{proof}

Now we address planar symmetry.  

\begin{proposition} \label{finite-sym-planes}
Let $S$ be a non-cylindrical algebraic surface. 
\begin{itemize}
\item [(1)] If $S$ is not a surface of revolution, then $S$ has finitely many symmetry planes, which intersect at least in a point. 
\item [(2)] $S$ is a surface of revolution if and only if $S$ has infinitely many symmetry planes.
\end{itemize}
\end{proposition}

\begin{proof} (1) By part (4) of Lemma \ref{composition}, the composition of two planar symmetries is either a translation, if the symmetry planes are parallel, or a rotation about their common line, if the symmetry planes are concurrent. Since by hypothesis $S$ is not cylindrical, by Proposition \ref{th-gen} every two symmetry planes must intersect at a line, which is an axis of rotation of $S$. Now by Proposition \ref{finite-rot-axes}, either the number of symmetry planes is finite, or there are infinitely many symmetry planes intersecting at a certain line $\ell$. However, in this second case these planes determine infinitely many different angles with each other. Since $S$ is invariant under the rotations about $\ell$ by any of these angles, by Corollary \ref{finiteang} $S$ must be a surface of revolution, which cannot happen by hypothesis. Finally, notice that all the symmetry planes must intersect because every two symmetry planes determine an axis of rotation of $S$, and by Proposition \ref{finite-rot-axes} all these axes intersect. 

(2) $(\Leftarrow)$ follows from statement (1). $(\Rightarrow)$ If $S$ is a surface of revolution then any plane containing the axis of revolution is a symmetry plane. 
\end{proof}

We finish this section with some observations about surfaces of revolution. If $S$ is a surface of revolution which is not a union of spheres, by Proposition \ref{rotaxis} the axis of revolution $\ell$ of $S$ is unique. Furthermore, $\ell$ can be computed by using the algorithms in \cite{AG14}, \cite{Vrseck}. It is clear, by construction, that the axis of revolution is a symmetry axis of $S$, and that any plane containing $\ell$ is a symmetry plane of $S$. Now the remaining symmetries of $S$ also follow, by construction, of the symmetries of a section of $S$ with any plane $\Pi$ containing $\ell$; such a section, that we denote by ${\mathcal D}$, is called a \emph{directrix} curve of $S$; notice that $S$ can be generated by rotating ${\mathcal D}$ around $\ell$. More precisely, we have the following results. In all the cases, the implication $(\Leftarrow)$ is a consequence of the fact that $S$ is constructed by rotating ${\mathcal D}$ about $\ell$. So we focus on $(\Rightarrow)$. 

\begin{proposition} \label{rev1}
Let $S$ be a surface of revolution, not a plane. $S$ has central symmetry iff ${\mathcal D}$ has central symmetry with respect to a point of $\ell$.
\end{proposition}

\begin{proof} We prove $(\Rightarrow)$. Suppose that $S$ is symmetric with respect to a point $P$. If $P\notin \ell$, then let $\Pi$ be the plane containing $P$ and $\ell$. The curve $S\cap\Pi$ inherits central symmetry with respect to $P$. Now let $P'\in S\cap\Pi$ be the symmetric point of $P$ with respect to $\ell$. Since $\ell$ is 
a symmetry axis of $S\cap\Pi$, we deduce that $S\cap\Pi$ is also symmetric with respect to $P'$. Therefore $S\cap\Pi$ has two different centers of symmetry, and hence $S\cap\Pi$ is a line containing $P,P'$. Since $S$ is a surface of revolution generated by this line, $S$ must be a plane, which is forbidden by hypothesis. So $P\in \ell$. Since the intersection curve of $S$ with any plane containing $\ell$ inherits central symmetry, the result follows.  
\end{proof}

\begin{proposition} \label{rev2}
Let $S$ be a surface of revolution, not a sphere or a plane, and let $\ell$ be the axis of revolution. A line $\tilde{\ell}\neq \ell$ is a symmetry axis of $S$ iff $\tilde{\ell}$ is the result of rotating around $\ell$ a symmetry axis of ${\mathcal D}$ which is normal to $\ell$.
\end{proposition}

\begin{proof}
We prove $(\Rightarrow)$. If $S$ is cylindrical then $S$ is a union of circular cylinders with the same axis, and the result follows. So assume that $S$ is not cylindrical. From Corollary \ref{finite-axial-sym}, $\tilde{\ell}$ and $\ell$ must intersect. Now suppose that $\tilde{\ell}$ is not perpendicular to $\ell$, and let $\Phi\neq 0,\frac{\pi}{2}$ be the angle between $\tilde{\ell},\ell$. Since $\ell$ is an axis of revolution of $S$, we have that $S$ is invariant under every rotation ${\mathcal R}_{\ell,\theta}$, with $\theta\in [0,2\pi)$. Furthermore, since $\tilde{\ell}$ is a symmetry axis, $S$ is invariant under the rotation ${\mathcal R}_{\tilde{\ell},\pi}$. So $S$ is invariant under the composition of ${\mathcal R}_{\ell,\theta}$, for any $\theta$, and ${\mathcal R}_{\tilde{\ell},\pi}$. By the statement (2) of Lemma \ref{composition}, the composition of these two rotations is another rotation about an axis $L$ concurrent with $\ell,\tilde{\ell}$, by an angle $\gamma$ corresponding to Equation \eqref{cos}. Now if $\Phi\neq \frac{\pi}{2}$, then since $\theta$ can take any value we get infinitely many values for $\gamma$ too. But then from Lemma \ref{angle} we deduce that $L$ is another axis of revolution, implying that $S$ has two different axes of revolution. However, by Proposition \ref{rotaxis} this means that $S$ is a sphere or a plane, which is forbidden by hypothesis. Therefore, $\tilde{\ell}$ is perpendicular to $\ell$. Finally, since $\tilde{\ell}$ is a symmetry axis of $S$, the section ${\mathcal G}$ of $S$ with the plane containing both $\tilde{\ell},\ell$ inherits the symmetry with respect to $\tilde{\ell}$, i.e. ${\mathcal G}$ is symmetric with respect to $\tilde{\ell}$. But ${\mathcal G}$ is the result of rotating ${\mathcal D}$ around $\ell$, and hence the result follows. 
\end{proof}

\begin{proposition} \label{rev3}
Let $S$ be a surface of revolution, not a sphere or a plane, let $\ell$ be the axis of revolution of $S$, and let $\Pi$ be a plane not containing $\ell$. The plane $\Pi$ is a symmetry plane of $S$ iff $\Pi$ is normal to $\ell$, and ${\mathcal D}$ is symmetric with respect to the intersection line of $\Pi$ and the plane containing ${\mathcal D}$.
\end{proposition}

\begin{proof}
We prove $(\Rightarrow)$. If $S$ is cylindrical then $S$ is a union of circular cylinders with the same axis, and the result follows. So assume that $S$ is not cylindrical. If $\Pi$ is parallel to $\ell$, then we can find another symmetry plane $\Pi^{\star}$, containing $\ell$, which is parallel to $\Pi$. Therefore $S$ is symmetric with respect to two parallel planes, namely $\Pi$ and $\Pi^{\star}$, which by statement (4) of Lemma \ref{composition} and Proposition \ref{th-gen} implies that $S$ is cylindrical. Since we are assuming that we are not in this case, $\Pi$ and $\ell$ intersect. If $\Pi$ is not normal to $\ell$, then the intersection of $\Pi$ with any plane $\tilde{\Pi}$ containing $\ell$ yields a line $\tilde{\ell}$. Furthermore, also by statement (4) of Lemma \ref{composition}, $S$ is invariant under a rotation about $\tilde{\ell}$ by an angle equal to twice the angle between $\Pi$ and $\tilde{\Pi}$. If we pick $\tilde{\Pi}$ such that the angle $\theta$ between $\Pi$ and $\tilde{\Pi}$ is not  $\theta=\frac{\pi}{m}$, with $m\in {\Bbb N}$, by Lemma \ref{angle} we deduce that the line $\tilde{\ell}=\Pi\cap \tilde{\Pi}$ is an axis of revolution of $S$. Since $\tilde{\ell}\neq \ell$, this implies that $S$ has two different axes of revolution, and therefore $S$ is either a plane or a sphere, which is forbidden by hypothesis. Therefore, $\Pi$ is perpendicular to $\ell$. Now the section of $S$ with any plane $\hat{\Pi}$ containing $\ell$ yields a curve which is symmetric with respect to the intersection of $\Pi$ and $\hat{\Pi}$. Thus the rest of the implication follows.
\end{proof}

\begin{corollary}\label{revolsym} Let $S$ be a surface of revolution, not a sphere or a plane.
\begin{itemize}
\item [(1)] If ${\mathcal D}$ has no symmetry axis perpendicular to $\ell$, then $S$ has just one symmetry axis (the axis of revolution), and the symmetry planes of $S$ are the planes containing $\ell$.
\item [(2)] $S$ has either one axis of symmetry (the axis of revolution) or infinitely many axes of symmetry.
\item [(3)] $S$ has infinitely many axes of symmetry iff $S$ has some symmetry plane not containing the axis of revolution. Furthermore, the symmetry axes of $S$ are the axis of revolution, and the intersections of the planes containing the axis of revolution with other symmetry planes not containing it.
\item [(4)] If $S$ has infinitely many axes of symmetry and $S$ is not cylindrical, all the symmetry axes of $S$ but one (the axis of revolution) lie on one plane, which is a symmetry plane of $S$.
\end{itemize}
\end{corollary}

In order to find the symmetries of a planar algebraic curve (${\mathcal D}$, in this case) one can apply for instance the results in \cite{Alcazar13, Alcazar.Hermoso13, LR08, LRTh}.

\section{Involutions of Polynomially Parametrized Surfaces.} \label{first-sec}

Throughout this section we will assume that $S$ is a non-cylindrical surface admitting a polynomial parametrization in the conditions formulated in Subsection \ref{prop-norm}. The special case of cylindrical surfaces will be treated in Section \ref{sec-cylindrical}.

Our goal is, first, to detect if $S$ exhibits central symmetry, axial symmetry or symmetry about a plane, and second, in the affirmative case, compute the elements of the symmetry (the symmetry center, the symmetry axes and the symmetry planes, respectively). The surface $S$ exhibits some of these symmetries if and only if there exists an affine mapping $f: \RR^3\longrightarrow \RR^3$, $f(x) = Qx + \bfb$, with $Q$ orthogonal, such that $f^2=\mbox{id}_{{\Bbb R}^3}$ and $f(S)=S$. Furthermore, since $\bfx$ is proper then $\bfx^{-1}$ exists and we have a mapping $\varphi: \RR^2 \rightarrow \RR^2$ making the following diagram commute:
\begin{equation}\label{eq:fundamentaldiagram}
\xymatrix{
S \ar[r]^{f} & S \\
\RR^2 \ar@{-->}[u]^{\bfx} \ar@{-->}[r]_{\varphi} & \RR^2 \ar@{-->}[u]_{\bfx}
}
\end{equation}

\begin{theorem} \label{th-fund}
Let $S$ be a surface properly, normally and polynomially parametrized by ${\bfx}$ and let $f:{\Bbb R}^3\to {\Bbb R}^3$ be a linear mapping $f(x)=Qx+\bfb$, with $Q$ orthogonal, such that $f(S)=S$. Then $f$ is an involution of $S$ if and only if there exists a unique mapping $\varphi:\RR^2\to \RR^2$ satisfying the following conditions: (1) $\bfx\circ \varphi=f \circ {\bfx}$; (2) $\varphi$ is linear affine; (3) $\varphi^2=\mbox{id}_{{\Bbb R}^2}$.
\end{theorem}

 \begin{proof} 
``$\Longrightarrow$'': Condition (1) must hold because ${\bfx}$ is invertible, and therefore we can define $\varphi=\bfx^{-1}\circ f \circ {\bfx}$. As for condition (2), we observe the following:

(i) {\it $\varphi(t,s)$ is a real, rational mapping:} since ${\bfx}$ is proper, ${\bfx}^{-1}$ is a real, rational mapping. So $\varphi={\bfx}^{-1}\circ f \circ {\bfx}$ is a composition of real rational mappings, and therefore $\varphi$ is also real rational itself. 

(ii) {\it $\varphi(t,s)$ is polynomial:} indeed, if $\varphi$ is not polynomial then we can find infinitely many (possibly complex) affine points in the $(t,s)$-plane such that the extension $\hat{\varphi}$ of $\varphi$ to the complex projective plane ${\Bbb P}^2({\Bbb C})$ maps them to points at infinity. Let $P$ be one of these points. Since $\bfx$ is polynomial, $P$ is mapped to an affine point $Q={\bfx}(P)$ on the surface. Following the diagram \eqref{eq:fundamentaldiagram}, the symmetry $f$ maps $Q$ to an affine point $Q'=f({\bfx}(P))$. Since $\bfx$ is normal, $Q'$ must be generated by some point in the parameter space. And since $Q'$ is affine and $\bfx$ is polynomial, $Q'$ must be generated by an affine pair $(t',s')$, and not by a point at infinity. Hence, $\varphi$ must be polynomial. 

(iii) {\it $\varphi(t,s)$ is linear affine:} let $\varphi(t,s)=(P(t,s),Q(t,s))$, with $P,Q$ polynomials, and let $k=\mbox{max}\{\deg(P),\deg(Q)\}$. We want to prove that $k=1$. For this purpose, let $L$ be a generic line of the plane $(t,s)$, i.e. $s=a+bt$ with $a,b$ generic. Since $L$ is generic, $\bfx(L)$ is a space curve contained in $S$ of degree $n$, where $n$ is the total degree of the parametrization $\bfx$. Additionally, since $f$ is an affine map we have that $(f\circ \bfx)(L)=\bfx(\varphi(L))$ is also a space curve of degree $n$. Now let \[\varphi(L)=\varphi(t,a+bt)=(P(t,a+bt),Q(t,a+bt))=(p(t),q(t)).\]If the degree of either $p(t)$ or $q(t)$ is not 1, again by the genericity of $L$ we have that $\bfx(\varphi(L))$ is a space curve of degree higher than $n$. So, $\mbox{deg}(p)=\mbox{deg}(q)=1$. But again because of the genericity of $a,b$, this implies that the degrees of $P(t,s), Q(t,s)$ must be one. This completes the proof of the condition (2). 

As for condition (3), we have that $\varphi^2=\bfx^{-1}\circ f^2 \circ \bfx$; since $f^2=\mbox{id}_{{\Bbb R}^3}$,  $\varphi^2=\mbox{id}_{{\Bbb R}^2}$ too. 

``$\Longleftarrow$'': since $f=\bfx \circ \varphi \circ \bfx^{-1}$, we get that $f^2=\bfx \circ \varphi^2 \circ \bfx^{-1}$; but $\varphi^2=\mbox{id}_{{\Bbb R}^2}$ and thus $f^2=\mbox{id}_{{\Bbb R}^3}$. 

\noindent The uniqueness of $\varphi$ follows also from the relationship $f=\bfx \circ \varphi \circ \bfx^{-1}$.
 \end{proof}

From Theorem \ref{th-fund}, any involution of $S$ is the result of lifting an involution of the plane, defined by $\varphi(t,s)$, to $S$ by means of the parametrization $\bfx$. Furthermore, if $S$ has involution symmetry then, also from Theorem \ref{th-fund}, we have that
\begin{equation} \label{reference-eq}
Q\bfx(t,s)+\bfb=\bfx(\varphi(t,s))
\end{equation}
for appropriate $Q,\bfb,\varphi(t,s)$, with $Q$ orthogonal. The main idea of our method is, first, to write all the parameters of $\varphi$ in terms of at most two of them, and then to write $Q,\bfb$ also in terms of these parameters. Then \eqref{reference-eq} gives rise to a bivariate polynomial system, whose solutions can be isolated by applying existing methods (see \cite{Aubry, Rou1, Rou2}). The consistency of the system guarantees the existence of symmetry.

The map $\varphi$ can be written as
\begin{equation*}
\varphi: \RR^2\longrightarrow \RR^2,\qquad \bft\longmapsto {\mathcal A}\bft + \bfc = \begin{bmatrix}a&b\\c&d \end{bmatrix}\begin{bmatrix}t\\s \end{bmatrix} + \begin{bmatrix}c_1\\c_2 \end{bmatrix},
\end{equation*}

\noindent where, from condition (3) in Theorem \ref{th-fund}, we get $\Delta=ad-bc\neq 0$. Therefore $\varphi(t,s)$ depends on 6 variables. The next lemma allows us to reduce the number of variables to at most 3.

\begin{lemma} \label{fundam}
The matrix ${\mathcal A}$ and the vector $\bfc$ satisfy one of the following:
\begin{itemize}
\item [(a)] ${\mathcal A} = -I$ and $\bfc\in \RR^2$,
\item [(b)] ${\mathcal A} = \begin{bmatrix} 1 & b\\ 0 &-1\end{bmatrix}, \ \bfc = c_2\begin{bmatrix}-b/2\\ 1\end{bmatrix}$,
\item [(c)] ${\mathcal A} = \begin{bmatrix}-1 & b\\ 0 & 1\end{bmatrix}, \ \bfc = \begin{bmatrix}c_1\\ 0\end{bmatrix}$,
\item [(d)] ${\mathcal A} = \begin{bmatrix} a & (1 - a^2)/c \\ c & - a\end{bmatrix},\ \bfc = c_2\begin{bmatrix} (a-1)/c \\1 \end{bmatrix}$, $c\neq0 $.
\end{itemize}
\end{lemma}

\begin{proof} Since $\bft = \varphi^2(\bft) = {\mathcal A}^2 \bft + {\mathcal A}\bfc + \bfc$, it follows that $({\mathcal A} + I)\big( ({\mathcal A}-I)\bft + \bfc \big) = 0$ for all $\bft$. Picking $\bft = 0$ shows that $\bfc\in\ker({\mathcal A}+I)$, and therefore $\bft = \varphi^2(\bft) = {\mathcal A}^2 \bft$ for all $\bft$. It follows that ${\mathcal A}^2 = I$ and the eigenvalues $\lambda, \mu$ of ${\mathcal A}$ are 1 or -1. Then
\[ \begin{bmatrix} 1 & 0\\ 0 & 1\end{bmatrix} = \begin{bmatrix} a & b\\ c & d \end{bmatrix}^2 = \begin{bmatrix}a^2 + bc & b(a + d)\\ c(a+d) & d^2 + bc \end{bmatrix}. \]
We distinguish two cases.

\emph{Case I: $a + d\neq 0$.} Then $b = c = 0$ and $a^2 = d^2 = 1$. Since $a + d\neq 0$, we get ${\mathcal A} = I$ or ${\mathcal A} = -I$. In the former case, $\bfc\in \ker ({\mathcal A} + I) = \ker(2I)$ yields $\bfc = 0$; this implies that $\varphi$ is the identity, which can be discarded as a trivial case. In the latter case any $\bfc\in \RR^2$ will satisfy $\varphi\circ \varphi = \mbox{Id}$.

\emph{Case II: $a + d = 0$.} Since $\mu + \lambda = \Tr {\mathcal A} = a + d = 0$, we find $\mu = \pm 1$ and $\lambda = \mp 1$. Since $-1 = \det {\mathcal A} = -a^2 - bc$, we also have $a^2 + bc = 1$. If $c = 0$, then $a^2 = 1$ and we obtain
\[ {\mathcal A} = \begin{bmatrix} 1 & b\\ 0 & -1\end{bmatrix}, \quad \bfc = c_2\begin{bmatrix}-b/2\\ 1\end{bmatrix},\quad\text{or}\quad
{\mathcal A} = \begin{bmatrix} -1 & b\\ 0 & 1\end{bmatrix}, \quad \bfc = \begin{bmatrix}c_1\\ 0\end{bmatrix}. \]
If $c\neq 0$, then $b = (1 - a^2)/c$ and
\[
{\mathcal A} = \begin{bmatrix} a & (1 - a^2)/c \\ c & - a\end{bmatrix},\qquad \bfc = c_2\begin{bmatrix} (a-1)/c \\1 \end{bmatrix}.
\]\end{proof}

Throughout the paper we will refer to the cases in the above lemma as cases (a), (b), (c), (d), respectively. In the first three cases, ${\mathcal A}$ and $\bfc$ depend on 2 variables; in the last case, they depend on 3 variables. Let us find some extra relationships, that will allow us to reduce the number of variables to two, also in the last case. In order to do this, we will make use of the first fundamental form of $\bfx$. Recall that this is a form defined in the tangent space of $S$ by means of the matrix:
\[{\bf I}_{\bfx}=\begin{bmatrix}
E & F \\
F & G
\end{bmatrix}=\begin{bmatrix} {\bfx}_t\cdot {\bfx}_t & {\bfx}_t\cdot {\bfx}_s \\ {\bfx}_t\cdot {\bfx}_s & {\bfx}_s\cdot {\bfx}_s \end{bmatrix}.
\]Now if $\xi$ is an isometry between two surfaces $S_1$ and $S_2$ then the first fundamental forms of $S_1,S_2$ are equal at corresponding points (see \cite[\S 4.2] {Docarmo}); i.e. if $P'=\xi(P)$, then ${\bf I}_{\xi \circ {\bfx}}(P')={\bf I}_{\bfx}(P)$. Notice that any symmetry is an isometry. Furthermore, if $f$ is a symmetry of $S$, then $f(S)=S$ is also parametrized by ${\bfx} \circ \varphi$. Hence for $P,P'\in S$ satisfying $P'=f(P)$, since by Theorem \ref{th-fund} $f\circ \bfx=\bfx \circ \varphi$, we have 
${\bf I}_{{\bfx} \circ \varphi}(P')={\bf I}_{\bfx}(P)$. Let $\tilde{\bfx}={\bfx}\circ {\varphi}$, and let
\[{\bf I}_{\tilde{\bfx}}=\begin{bmatrix}
\tilde{E} & \tilde{F} \\
\tilde{F} & \tilde{G}
\end{bmatrix}=\begin{bmatrix} \tilde{\bfx}_t\cdot \tilde{\bfx}_t & \tilde{\bfx}_t\cdot \tilde{\bfx}_s \\ \tilde{\bfx}_t\cdot \tilde{\bfx}_s & \tilde{\bfx}_s\cdot \tilde{\bfx}_s \end{bmatrix}.
\]Since $\varphi({\bf 0})=(c_1,c_2)={\bf c}$, in particular we get that
\begin{equation} \label{rel-1}
E({\bf 0})=\tilde{E}({\bf c}), \mbox{ }F({\bf 0})=\tilde{F}({\bf c}),\mbox{ }G({\bf 0})=\tilde{G}({\bf c}).
\end{equation}
In order to exploit the above relationships we need to write $\tilde{E}({\bf c}),\tilde{F}({\bf c}),\tilde{G}({\bf c})$ in terms of $E({\bf c}),F({\bf c}),G({\bf c})$. Now we observe that
\[\nabla(\tilde{\bfx})=\nabla({\bfx}\circ \varphi)=\begin{bmatrix} \tilde{\bfx}_t \\ \tilde{\bfx}_s \end{bmatrix}=\begin{bmatrix} (\bfx \circ \varphi)_t\\ (\bfx \circ \varphi)_s \end{bmatrix}=\begin{bmatrix} a & c \\ b & d \end{bmatrix}\cdot \begin{bmatrix} {\bfx}_t \\ {\bfx}_s\end{bmatrix}=\begin{bmatrix} a{\bfx}_t+c{\bfx}_s\\ b{\bfx}_t+d{\bfx}_s\end{bmatrix}.\]Using this together with (\ref{rel-1}), we reach the relationships:

\begin{equation} \label{relations}
\begin{array}{rcl}
E({\bf 0})&=&E({\bf c})\cdot a^2+2F({\bf c})\cdot ac+G({\bf c})\cdot c^2\\
F({\bf 0})&=&E({\bf c})\cdot ab+F({\bf c})\cdot (ad+bc)+G({\bf c})\cdot cd\\
G({\bf 0})&=&E({\bf c})\cdot b^2+2F({\bf c})\cdot bd+G({\bf c})\cdot d^2
\end{array}.
\end{equation}

\noindent Finally, for the sake of convenience, let us denote
\[
E({\bf 0})=A, \mbox{ }F({\bf 0})=B, \mbox{ }G({\bf 0})=C.
\]Also, let us recall the notation $\Delta=ad-bc\neq 0$. Then we can solve \eqref{relations} for $E({\bf c})$, $F({\bf c})$ and $G({\bf c})$, to get
\begin{equation} \label{add}
\begin{array}{rcl}
E({\bf c})&=&\displaystyle{\frac{Cc^2+Ad^2-2Bcd}{\Delta^2}}\\
F({\bf c})&=&\displaystyle{\frac{B(bc+ad)-Abd-Cac}{\Delta^2}}\\
G({\bf c})&=&\displaystyle{\frac{Ab^2-2Bab+Ca^2}{\Delta^2}}
\end{array}
\end{equation}
Applying the
relations (\ref{add}) to each of the four cases in Lemma \ref{fundam} we reach the following result, where we are left with 2 variables in all the cases.

\begin{lemma} \label{config-complete}
The possible configurations for ${\mathcal A}$, $\bfb$ are:
\begin{itemize}
\item [(a)] ${\mathcal A} = -I$, $E({\bf c})=A$, $F({\bf c})=B$, $G({\bf c})=C$.
\item [(b)] ${\mathcal A} = \begin{bmatrix} 1 & b\\ 0 &-1\end{bmatrix}, \ \bfc = c_2\begin{bmatrix}-b/2\\ 1\end{bmatrix}$, $E({\bf c})=A$, $F({\bf c})=Ab-B$, $G({\bf c})=b^2-2Bb+C$.
\item [(c)] ${\mathcal A} = \begin{bmatrix}-1 & b\\ 0 & 1\end{bmatrix}, \ \bfc = \begin{bmatrix}c_1\\ 0\end{bmatrix}$, $E({\bf c})=A$, $F({\bf c})=-Ab-B$, $G({\bf c})=Ab^2+2Bb+C$.
\item [(d)] ${\mathcal A} = \begin{bmatrix} a & (1 - a^2)/c \\ c & - a\end{bmatrix},\ \bfc = c_2\begin{bmatrix} (a-1)/c \\1 \end{bmatrix}$, $c\neq0 $, with two possible subcases:
    \begin{itemize}
\item [(d.1)] $\bfc=0$.
\item [(d.2)] $\bfc\neq 0$, $a=\displaystyle{\frac{-c_1^2E({\bf c})-c_1c_2[F({\bf c})-B]+Cc_2^2}{Ac_1^2+2Bc_1c_2+Cc_2^2}}$, where $Ac_1^2+2Bc_1c_2+Cc_2^2\neq 0$.
\end{itemize}
\end{itemize}
\end{lemma}

\begin{proof} The first three configurations follow in a very straightforward way from the relations (\ref{add}) and Lemma \ref{fundam}. So let us deduce the last one. In the case (d) of Lemma \ref{fundam}, by using the first relation in \eqref{add} we get \[E({\bf c})=Aa^2+2Bac+Cc^2.\]Multiplying by $c_1^2$, and taking into account that $c_1c=c_2(a-1)$, we get that
\begin{equation}\label{l1-1}
c_1^2E({\bf c})=(Cc_2^2+2Bc_1c_2+Ac_1^2)\cdot a^2-(2c_2C+2c_1c_2B)\cdot a +Cc_2^2.
\end{equation}
On the other hand, from the second relation in (\ref{add}), and since $b=\frac{1-a^2}{c}$ (notice that $c\neq 0$) and $d=-a$,
\[F({\bf c})=A\frac{1-a^2}{c}a+B(1-2a^2)-Cac.\]Furthermore, since $c\cdot c_1=c_2\cdot(a-1)$, after multiplying the above equation by $c_1\cdot c_2$ we can write
\[c_1c_2F({\bf c})=-Ac_1^2a(1+a)+Bc_1c_2(1-2a^2)-Cc_2^2 a(a-1),\]and hence we get
\begin{equation}\label{l1-2}
c_1c_2F({\bf c})=(-Ac_1^2-2Bc_1c_2-Cc_2)\cdot a^2-(Ac_1^2+Cc_2^2)\cdot a+Bc_1c_2.
\end{equation}
By adding up (\ref{l1-1}) and (\ref{l1-2}), the terms in $a^2$ cancel, and we obtain
\[
(Ac_1^2+2Bc_1c_2+Cc_2^2)\cdot a =-c_1^2E({\bf c})-c_1c_2[F({\bf c})-B]+Cc_2^2.
\]
Notice that \[Ac_1^2+2Bc_1c_2+Cc_2^2={\bf c}^T\cdot I_{{\bfx}(0)}\cdot {\bf c}.\]Therefore, since $\bfx(0)$ is regular by hypothesis (see the end of Subsection \ref{prop-norm}) and because of the positive-definiteness of the first fundamental form, whenever ${\bf c}\neq {\bf 0}$ we have that $Ac_1^2+2Bc_1c_2+Cc_2^2\neq 0$. Then the result follows.
\end{proof}

Hence, in case (a) we are left with the variables $c_1,c_2$; in case (b) we are left with $b,c_2$; in case (c) we are left with $b,c_1$. In case (d), if $\bfc=0$ then we are left with $a,c$; if $\bfc\neq 0$, we distinguish two subcases: (i) if $a=1$ then $c_1=b=0$, and we are left with $a,c$; (ii) if $a\neq 1$ then $c_1\neq 0$ and we can write $c=c_2\cdot \frac{a-1}{c_1}$, so we are left with $c_1,c_2$.

Finally, we need to write $Q,\bfb$ in $f(x)=Qx+\bfb$ in terms of the parameters of $\varphi$. For this purpose, by differentiating \eqref{reference-eq} with respect to $t,s$ we get:
\begin{equation}\label{d-2}
\begin{array}{c}
Q\cdot {\bfx}_t(t,s)={\bfx}_t(\varphi(t,s))\cdot a +{\bfx}_s(\varphi(t,s))\cdot c\\
Q\cdot {\bfx}_s(t,s)={\bfx}_t(\varphi(t,s))\cdot b +{\bfx}_s(\varphi(t,s))\cdot d.
\end{array}
\end{equation}
Evaluating \eqref{d-2} at $(t,s)=(0,0)$ yields
\begin{equation}\label{d-3}
\begin{array}{c}
Q\cdot {\bfx}_t(0,0)={\bfx}_t({\bf c})\cdot a +{\bfx}_s({\bf c})\cdot c\\
Q\cdot {\bfx}_s(0,0)={\bfx}_t({\bf c})\cdot b +{\bfx}_s({\bf c})\cdot d.
\end{array}
\end{equation}
Additionally, the normal line to $S$ at ${\bfx}(0,0)$ is parallel to ${\bfx}_t(0,0)\times {\bfx}_s(0,0)\neq {\bf 0}$ (recall that $\bfx(0)$ is regular by hypothesis). From \eqref{d-3} we have 
\begin{equation}\label{d-4}
Q\cdot ({\bfx}_t(0,0) \times {\bfx}_s(0,0))=\mbox{det}(Q)\cdot \Delta\cdot ({\bfx}_t \times {\bfx}_s)({\bf c}),
\end{equation}
 where $\mbox{det}(Q)=\pm 1$ depending on whether $Q$ preserves orientation (axial symmetries) or not (central and planar symmetries). So by using \eqref{d-3} and \eqref{d-4} we can derive the matrix $Q$ from its action on ${\bfx}_t(0,0)$, ${\bfx}_s(0,0)$ and ${\bfx}_t(0,0)\times {\bfx}_s(0,0)$. Multiplying $Q$ by the matrix \[M=[{\bfx}_t(0,0), {\bfx}_s(0,0),{\bfx}_t(0,0)\times {\bfx}_s(0,0)],\]yields the matrix
\begin{equation}\label{thematrix}
L=[{\bfx}_t({\bf c})\cdot a +{\bfx}_s({\bf c})\cdot c,{\bfx}_t({\bf c})\cdot b +{\bfx}_s({\bf c})\cdot d,\mbox{det}(Q)\cdot \Delta\cdot ({\bfx}_t \times {\bfx}_s)({\bf c})].
\end{equation}
Hence $Q=LM^{-1}$, and therefore the elements of $Q$ are written in terms of the parameters of $\varphi$. By evaluating  \eqref{reference-eq} at $t=0$, we deduce that
\begin{equation} \label{elbe}
\bfb=\bfx(\bfc)-Q\bfx(0).
\end{equation}

\subsection{Detection of direct involutions.} \label{subse-orient-preserve}

In order to detect orientation-preserving involutions, i.e. axial symmetries, we must fix $\mbox{det}(Q)=1$ in \eqref{thematrix}, and then check if each polynomial system obtained from \eqref{reference-eq} for each possible configuration of ${\mathcal A}$, $\bfb$ (see Lemma \ref{config-complete}) has real solutions. 

If one gets infinitely many solutions then from Corollary \ref{finite-axial-sym} $S$ is a surface of revolution. In this case, the symmetry axes of $S$, other than the axis of revolution, can be computed from the section of $S$ with a plane containing the axis of revolution (see Proposition \ref{rev2} and Corollary \ref{revolsym}).

Now let us focus on the case when we get finitely many solutions, and let us see how to find the symmetry axes in that case. Since each symmetry axis $\ell$ is the set of fixed points of a symmetry $f(x)=Qx+\bfb$, once $Q, \bfb$ have been determined one can find $\ell$ as the solution set of the system $(Q-I)x=-\bfb$. Observe that the direction of $\ell$ corresponds to the eigenspace associated with $\lambda=1$, which is an eigenvalue of $Q$. However, there is an alternative method to find $\ell$, based on the analysis of the involution $\varphi$ of the plane that gives rise to $f(x)$. For this purpose, we observe first that the fixed points of $\varphi$ can be found by solving $({\mathcal A}-I)\bft + \bfc= {\bf 0}$. Therefore we have the following result, that can be deduced after easy calculations.

\begin{lemma} \label{fixedpoints}
The following statements are true:
\begin{itemize}
\item [(i)] In case $(a)$, $\varphi$ has just one fixed point, namely $\bfc/2$.
\item [(ii)] In case $(b)$, $\varphi$ has: (i) one fixed point, namely $(c_2/2,c_2/2)$, if $b\neq 0$; (ii) a line of fixed points, namely $s=\frac{c_2}{2}$, if $b=0$.
\item [(iii)] In case $(c)$, $\varphi(t,s)$ has a line of fixed points, namely $t=\frac{b}{2}s+\frac{c_1}{2}$.
\item [(iv)] In case $(d)$, $\varphi(t,s)$ has a line of fixed points, namely $t=\frac{1}{c}(a+1)s-\frac{c_2}{c}$.
\end{itemize}
\end{lemma}

Notice that the set of fixed points ${\mathcal M}$ of $\varphi$ is nonempty in all the cases. Since $\bfx(\varphi(t,s))$ is the symmetric point of $\bfx(t,s)$, any fixed point of $\varphi$ leads to a fixed point of $f(x)$. The converse is not necessarily true: indeed, on one hand $f(x)$, as a mapping from $\RR^3$ to $\RR^3$, can have fixed points that do not belong to $S$. Additionally, if $P\in S$ is a fixed point of $f(x)$ reached by the parametrization $\bfx(t,s)$ then either $P\in {\mathcal M}$, or $P$ is a self-intersection of $S$. Now we have the following result.

\begin{proposition} \label{ax-elements}
Let $S$ be polynomially, properly and normally parametrized, and let $f(x)=Qx+\bfb$ be an axial symmetry of $S$ with symmetry axis $\ell$. Also, let ${\mathcal M}$ be the set of fixed points of the mapping $\varphi$ corresponding to $f$.  
\begin{itemize}
\item [(i)] If $\bfx({\mathcal M})$ is a straight line, then $\bfx({\mathcal M})=\ell$.
\item [(ii)] If $\bfx({\mathcal M})$ is a regular point $P$, then $\ell$ is normal to $S$ through $P$.
\end{itemize}
\end{proposition}

\begin{proof} From Lemma \ref{fixedpoints} and since $\bfx({\mathcal M})$ is included in the set of fixed points of $S$ with respect to the symmetry, which is at most a straight line, $\bfx({\mathcal M})$ is either a point or a straight line. If $\bfx({\mathcal M})$ is a straight line, then $\bfx({\mathcal M})$ coincides with the set of fixed points of $f$, i.e. $\bfx({\mathcal M})=\ell$, and (i) holds. Now let us see (ii). In order to prove this, observe that $f(x)=Qx+\bfb$ induces a symmetry of the same kind ${\bf n}_{f(\bfx)}=Q\cdot {\bf n}_{\bfx}$ between the normal vectors to $S$ at corresponding points $\bfx(t,s)$ and $f(\bfx(t,s))$. Since $P$ is regular by hypothesis, the normal vector to $S$ at $P$, ${\bf n}_P$, is well-defined. Now since $f(P)=P$ then $(Q-I)\cdot {\bf n}_P={\bf 0}$. Therefore ${\bf n}_P$ is an eigenvector of $Q$, associated with the eigenvalue $\lambda=1$, and hence its direction coincides with that of $\ell$.
\end{proof}

Notice that Proposition \ref{ax-elements} is not applicable when $\bfx({\mathcal M})$ reduces to a singular point of $S$. In that case, we compute $\ell$ as the solution set of $(Q-I)x=-\bfb$.

\subsection{Detection of opposite involutions.} \label{subse-orient-reverse}

Here we have central and planar symmetries. In order to detect them, one sets $\mbox{det}(Q)=-1$ in \eqref{thematrix}, and, as in Section \ref{subse-orient-preserve}, one must check whether or not the polynomial systems obtained from \eqref{reference-eq} for each possible configuration of ${\mathcal A}$, $\bfb$ (see Lemma \ref{config-complete}) have real solutions. In order to distinguish if the symmetry is central or planar, and also to find the elements of the symmetry (the symmetry center or the symmetry plane), one can compute $Q$, $\bfb$, and then study the solution set of $(Q-I)x=-\bfb$. In particular, an opposite involution $f(x)$ is a central symmetry if and only if the set of fixed points reduces to a point, i.e. iff $\mbox{det}(Q-I)\neq 0$; furthermore, in that case the fixed point is the symmetry center. If $\mbox{det}(Q-I)=0$, then the involution has a plane of fixed points, and that plane is a symmetry plane $\Pi$. Observe that $\Pi$ corresponds to the eigenspace of $Q$ associated with the eigenvalue $\lambda=1$.

However, as in Subsection \ref{subse-orient-preserve}, by exploiting Lemma \ref{fixedpoints} we arrive to an alternative method.

\begin{proposition} \label{opp-elements}
Let $S$ be polynomially, properly and normally parametrized, and let $f(x)=Qx+\bfb$ be an opposite involution of $S$. Also, let ${\mathcal M}$ be the set of fixed points of the mapping $\varphi$ corresponding to $f$.  
\begin{itemize}
\item [(i)] If $\bfx({\mathcal M})$ reduces to a regular point $P$, then $f(x)$ is a central symmetry and $P$ is the symmetry center.
\item [(ii)] If $\bfx({\mathcal M})$ is not a point, then $f(x)$ is a planar symmetry, and the symmetry plane $\Pi$ contains $\bfx({\mathcal M})$. Furthermore, if $\bfx({\mathcal M})$ is a straight line containing some regular point of $S$, then $\Pi$ is the plane defined by $\bfx({\mathcal M})$, and the normal lines to $S$ at the regular points of $\bfx({\mathcal M})$.  
\end{itemize}
\end{proposition}

\begin{proof} Let us see (i). If $f$ is a planar symmetry, by reasoning as in statement (ii) of Proposition \ref{ax-elements} we deduce that the normal vector to $S$ at $P$, ${\bf n}_P$, is an eigenvector of $Q$, and therefore the normal line $N_P$ to $S$ at $P$ is contained in the symmetry plane $\Pi$. Since $P$ is regular and $\Pi$ contains $N_P$, $\Pi$ intersects $S$ at a curve ${\mathcal C}$, which is a curve of fixed points contained in $S$. Since $\bfx({\mathcal M})=\{P\}$, the points of ${\mathcal C}-\{P\}$ must be singular. But this is a contradiction, because since $P$ is regular by hypothesis there must be a neighborhood $E_p$ of $P$ such that $E_p\cap S$ is regular. So $f$ must be a central symmetry, and the symmetry center is $P$. Now let us see (ii). The first part is clear. So assume that $\bfx({\mathcal M})$ is a straight line ${\mathcal L}$. This line is contained in the symmetry plane, $\Pi$. Furthermore, if ${\mathcal L}$ contains some regular point $P$, as before we have that the normal vector to $S$ at $P$, ${\bf n}_P$, is also contained in $\Pi$. Hence, the result follows.
\end{proof}

Figure 1 illustrates the statement (ii) of Proposition \ref{opp-elements}. Notice that if $\bfx({\mathcal M})$ is entirely contained in the singular locus of $S$, then Proposition \ref{opp-elements} is not applicable, and we need to analyze the solution set of $(Q-I)x=-\bfb$.

\begin{figure}\label{normalsec}
\begin{center}
  \includegraphics[scale=0.4]{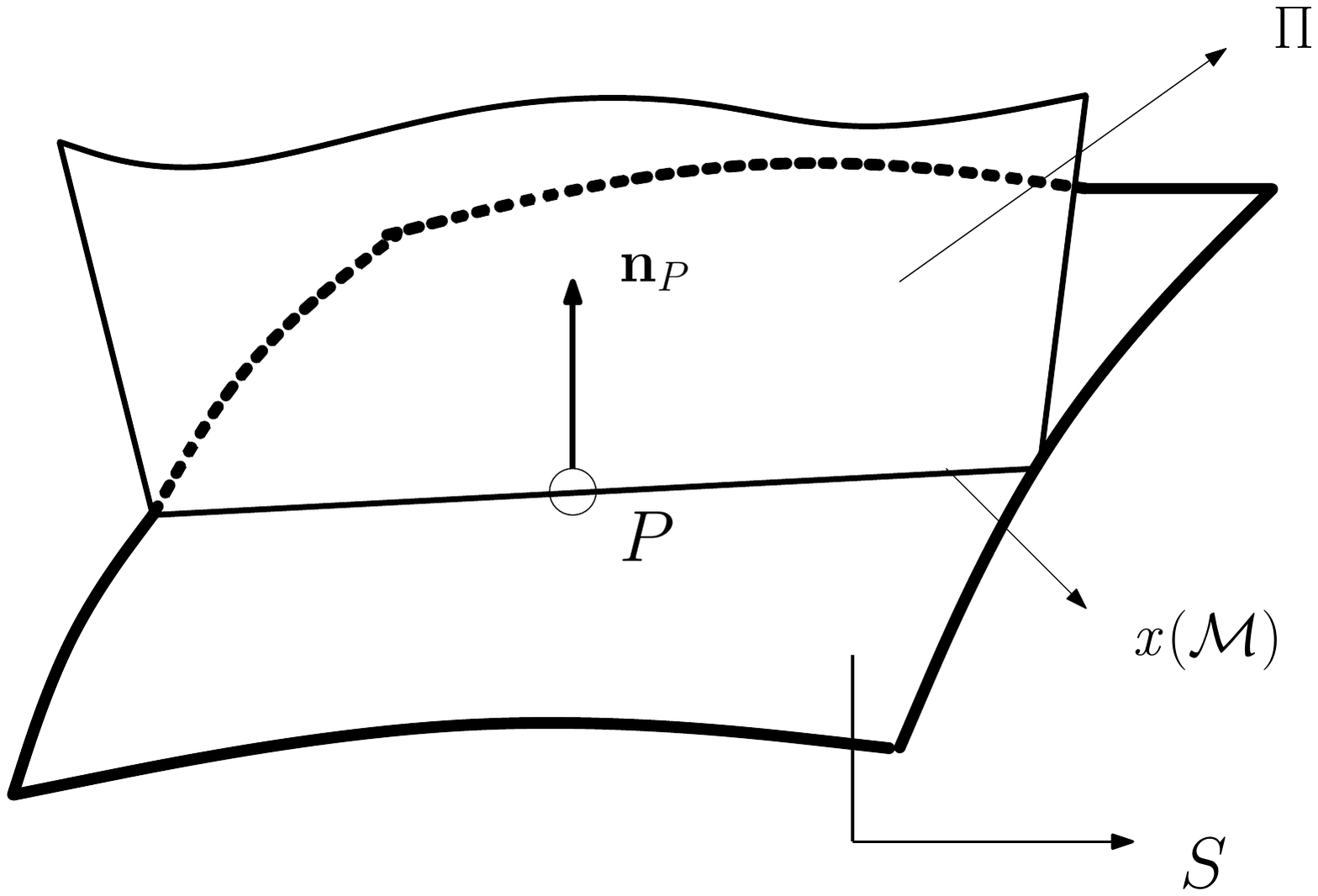} 
 \caption{Illustrating Proposition \ref{opp-elements}, (ii).}
\end{center}
\end{figure}

Notice that from Proposition \ref{several} the symmetry center, if it exists, is unique. Therefore, we can have at most one central inversion leaving $S$ invariant. However, from Proposition \ref{finite-sym-planes} we might have infinitely many symmetry planes, implying that $S$ is a surface of revolution. In that case, any plane containing the axis of revolution is a symmetry plane. Furthermore, $S$ could also have other symmetry planes, normal to the axis of revolution (see Proposition \ref{rev3} and Corollary \ref{revolsym}).

\subsection{Summary of the algorithm} \label{subsec-sum}

The following algorithm {\tt SymSurf} follows from the ideas in this section. 

\begin{algorithm}
\begin{algorithmic}[1]
\REQUIRE A proper and normal parametrization $\bfx:\RR^2\dashrightarrow \RR^3$ of a non-cylindrical surface $S$, where $\bfx(0)$ is a regular point. 
\ENSURE The involutions leaving $S$ invariant, and their characteristic elements. 
\STATE \emph{Direct involutions:} for each configuration in Lemma \ref{config-complete} do:
\STATE Find $Q$, $\bfb$ in terms of the parameters of $\varphi(t,s)$ from (\ref{thematrix}) and (\ref{elbe}). 
\STATE Derive the (in general, bivariate) polynomial system in the parameters of $\varphi(t,s)$ from equation (\ref{reference-eq}), taking into account $\mbox{det}(Q)=1$.
\STATE Check whether or not the polynomial system has real solutions. 
\STATE Derive the characteristic elements of the involution from Proposition \ref{ax-elements} or by solving the system $(Q-I)x=-\bfb$. If the system has infinitely many real solutions, return {\it The surface is a revolution surface}.
\STATE \emph{opposite involutions:} proceed in an analogous way, using Proposition \ref{opp-elements} instead of Proposition \ref{ax-elements}.
\end{algorithmic}
\caption*{{\bf Algorithm} {\tt SymSurf}}
\end{algorithm}

\section{The case of cylindrical surfaces} \label{sec-cylindrical}

If $S$ is a rational (not necessarily polynomial) surface, one can detect whether or not it is cylindrical by applying the results of \cite{Shen}. Furthermore, in that case one can also find \cite{Shen} a rational parametrization of $S$ of the form
\[{\bf y}(t,\lambda)=w(t)+\lambda{\bf v},\]where ${\bf v}$ denotes the direction of the generatrices of $S$. A first observation is that any plane normal to ${\bf v}$ is a symmetry plane; therefore, in this case we always have infinitely many symmetry planes. The other involutions leaving $S$ invariant can be found by analyzing a normal section of the surface. Indeed, let $\Pi \equiv Ax+By+Cz+D=0$ be a plane normal to the direction ${\bf v}$, and let ${\mathcal E}=S\cap \Pi$ be the normal section of $S$ corresponding to $\Pi$. By plugging the parametrization ${\bf y}(t,\lambda)$ into the equation of $\Pi$, we can 
solve for $\lambda$ to get $\lambda=\lambda(t)$; then, by substituting $\lambda(t)$ back into ${\bf y}(t,\lambda)$ we get a rational parametrization of ${\mathcal E}$. Now $S$ has axial symmetry if and only if ${\mathcal E}$ has central symmetry, and the symmetry axis is normal to $\Pi$ through the symmetry center of ${\mathcal E}$. Additionally, $S$ is symmetric w.r.t. a plane if and only if ${\mathcal E}$ is symmetric with respect to a line, and the symmetry plane is normal to $\Pi$ through the symmetry axis of ${\mathcal E}$. In order to determine the symmetries of ${\mathcal E}$ we can use the algorithms in \cite{Alcazar13, AHM13, Alcazar.Hermoso13}.

\section{Experimentation and implementation.} \label{sec-perf}

In this section we present some examples, and we report on the complexity and practical performance of the algorithm. The properness of the parametrizations tested here can be examined by using the techniques in \cite{PDSS02}. Furthermore, all the tested parametrizations satisfy the hypotheses in Corollary 3.15 of \cite{PSV}, which gives a sufficient condition for normality.

\subsection{An example: finding the involutions of an Enneper surface.}

Consider the Enneper surface $S$, a minimal surface of degree 9, parametrized by
\[\bfx(t,s)=(-s^3+3st^2+3s,3s^2t-t^3+3t,3s^2-3t^2).\]We explore first the direct symmetries of the surface. In order to do this, we have to test each case in Lemma \ref{config-complete}. Case (a) succeeds and yields the values $c_1=0$, $c_2=0$. From Lemma \ref{fixedpoints} it follows that the corresponding $\varphi(t,s)$ has just one fixed point, namely $(0,0)$. We get that $\bfx(0,0)=\bf{0}$; since $\bfx_t(0,0)\times \bfx_s(0,0)=(0,0,-9)\neq {\bf 0}$, we have that $\bfx(0,0)$ is a regular point. Therefore, from Proposition \ref{ax-elements} the $z$-axis is a symmetry axis of $S$. Case (d.1) also succeeds, and yields two solutions, namely $\{a=0,c=-1\}$ and $\{a=0,c=1\}$. In the first case, $\varphi(t,s)$ has a line of fixed points, $t=s$, which parametrizes the line $\{x-y=0,z=0\}$, contained in the surface. In the second case, $\varphi(t,s)$ has also a line of fixed points, $t=-s$, which gives rise to the line $\{x+y=0,z=0\}$, also contained in the surface. So we get three symmetry axes, which are plotted in Fig. 2 (see the bottom row; each plotting shows a different perspective of the surface and its symmetry axes).

As for the opposite symmetries of the surface, the case (b) yields the solution $\{b=0$, $c_2=0\}$. From Lemma \ref{fixedpoints}, the corresponding $\varphi(t,s)$ has a line of fixed points, namely $s=0$. Since $\bfx(t,0)=(0,-t^3+3t,-3t^2)$, which is a planar curve contained in the plane $x=0$, from Proposition \ref{opp-elements} we deduce that $x=0$ is a symmetry plane of $S$. In the case (c) we also get a solution, $\{b=0$, $c_1=0\}$; here, $t=0$ is the line of fixed points of $\varphi(t,s)$, and we have $\bfx(0,s)=(-s^3+s,0,3s^2)$, which is a planar curve contained in the plane $y=0$; so we get a planar symmetry too, with respect to the plane $y=0$ this time. Notice that the $z$-axis is precisely the intersection of the symmetry planes $x=0$, $y=0$. The symmetry planes of the surface are shown in Fig. 2 (see the top row, right and left).

\begin{center}
\begin{figure}
$$\hspace{-2.5 cm}\begin{array}{ccc}
\includegraphics[scale=0.4]{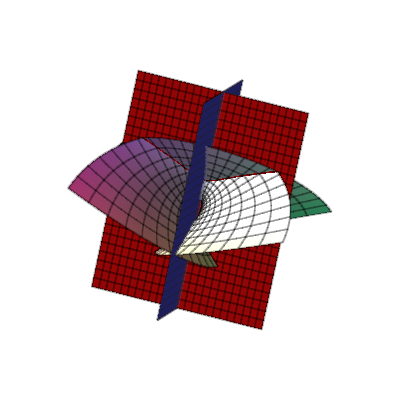} &   \includegraphics[scale=0.4]{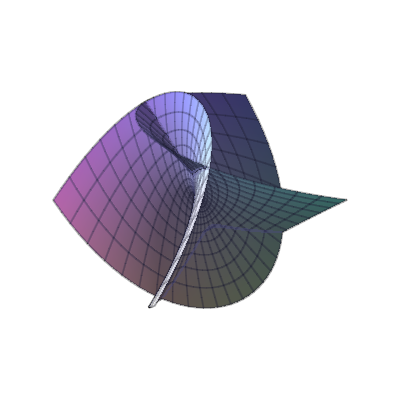} & \includegraphics[scale=0.4]{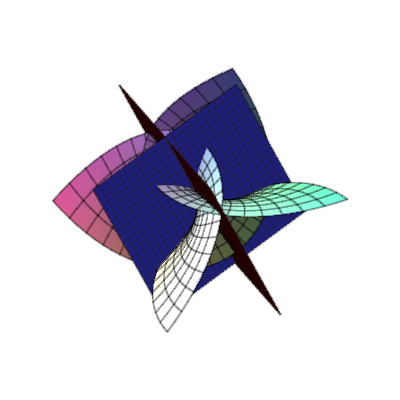}\\
\includegraphics[scale=0.4]{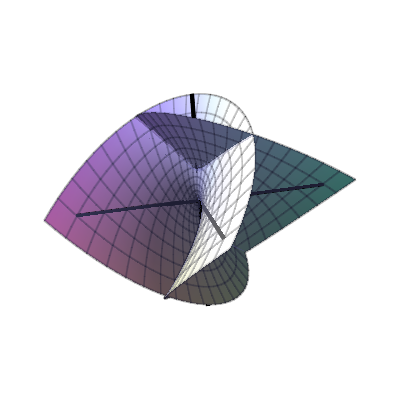} & \includegraphics[scale=0.4]{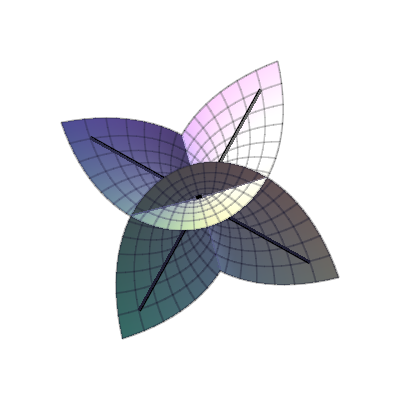} & \includegraphics[scale=0.4]{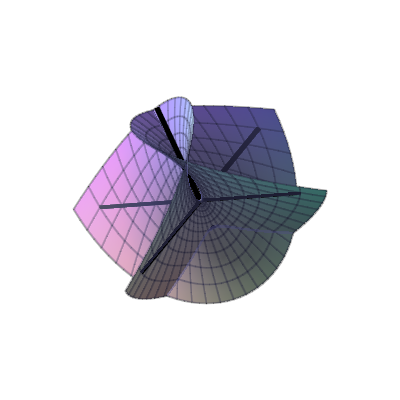}
 \end{array}$$
 \caption{An Enneper surface (top row, center), with symmetry planes (top row, right, left), and symmetry axes (bottom row)}
\end{figure}
\end{center}

\subsection{An example: finding the involutions of a circular paraboloid.}
Consider the circular paraboloid $S$, parametrized as 
\[\bfx(t,s)=(t,s,t^2+s^2)\]When looking for direct symmetries we observe that only case (a) succeeds, yielding the solution $\{c_1=c_2=0\}$. As in Example 1, we observe that 
$\bfx(0,0)=\bf{0}$; since $\bfx_t(0,0)\times \bfx_s(0,0)=(0,0,1)\neq {\bf 0}$, we get an axial symmetry with respect to the $z$-axis. As for the opposite symmetries, case (b) yields the solution $\{b=0, c_2=0\}$; the line of fixed points of the corresponding $\varphi(t,s)$ is $s=0$, which gives rise to the curve $(t,0,t^2)$. This is a planar curve contained in the plane $y=0$, which is therefore a symmetry plane of $S$. Case (c) also succeeds, yielding $\{b=0, c_1=0\}$. The line of fixed points of $\varphi(t,s)$ is then $t=0$, which gives rise to the curve $(0,s,s^2)$. Since this curve is contained in the plane $x=0$, we deduce that $x=0$ is another symmetry plane of $S$. Finally, case (d.1) succeeds too, but here we obtain infinitely many real solutions, which satisfy $a^2+c^2=1$. Therefore, by Proposition \ref{finite-sym-planes} we detect that $S$ is a surface of revolution. Furthermore, we observe that the lines of fixed points of the corresponding $\varphi(t,s)$'s are $t=\frac{a+1}{c}s$. A generic line of this family gives rise to the curve \[\left((a+1)s/c, s, s^2+(a+1)^2s^2/c^2\right),\]which belongs to the plane $x-\frac{a+1}{c}y=0$. So we deduce that the remaining symmetry planes of $S$ are $x-ky=0$, with $k\in {\Bbb Z}$. All these planes contain the $z$-axis, which therefore corresponds to the axis of revolution. Notice also that the axis of revolution is in particular a symmetry axis of $S$. The circular paraboloid is plotted in Fig. 3. At the right we have plotted, in thick line, the $z$-axis, which is the only symmetry axis of the surface; the $x$-axis and the $y$-axis are also plotted for reference, but they are not symmetry axes of $S$. At the left we have plotted two of the (infinitely many) symmetry planes of $S$, which intersect in the $z$-axis.

\begin{center}
\begin{figure}
$$\hspace{-1 cm}\begin{array}{cc}
\includegraphics[scale=0.5]{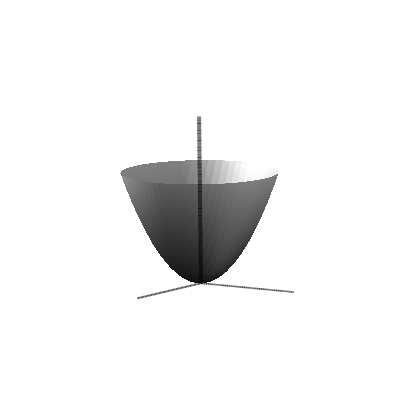} & \includegraphics[scale=0.5]{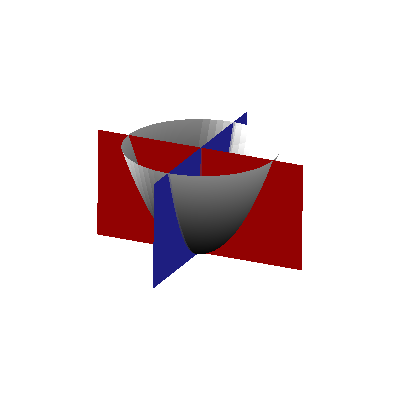}   
\end{array}$$
 \caption{Symmetry axis (left, in thick line) and two symmetry planes (right) of the circular paraboloid $(t,s,t^2+s^2)$.}
\end{figure}
\end{center}

\subsection{Observations on the complexity.} 

Let us consider the complexity of Algorithm {\tt SymSurf}. For this purpose we will analyze the case of direct involutions, and we will focus on the case (d.2) of 
Lemma \ref{config-complete}, which is the most demanding one. The complexity is not modified when one includes the other cases of Lemma \ref{config-complete}, or opposite involutions. Throughout this section, in addition to the standard \emph{Big Oh} notation $\OOO$, we use the \emph{Soft Oh} notation $\tOOO$ to ignore any logarithmic factors in the complexity analysis. Also, here we will speak about the ``degree of a rational function" to mean the maximum of the total degrees of the numerator and the denominator of the function. 

Let us consider first Step 2, i.e. the construction of $Q, \bfb$. Denoting by $d$ the total degree of $\bfx$, we observe that $\bfx_t, \bfx_s$ have total degrees bounded by $d-1$, and therefore the degree of $\bfx_t \times \bfx_s$ is bounded by $2d-2$. In the case (d.2) we get 
\[a=\displaystyle{\frac{-c_1^2E({\bf c})-c_1c_2[F({\bf c})-B]+Cc_2^2}{Ac_1^2+2Bc_1c_2+Cc_2^2}}, b=\frac{-(1+a)c_1}{c_2}, c=\frac{c_2(a-1)}{c_1}, d=-a. \]
Since $E(\bfc)=\bfx_t(\bfc)\cdot \bfx_t(\bfc)$ the degree of $E(\bfc)$ is bounded by $2d-2$; similarly for $F(\bfc)$. Therefore, $a$ is a rational function in $c_1,c_2$ with degree bounded by $2d$. We observe that $b,c,d$ are rational functions too, with degrees $\OOO(d)$. By using the expression \eqref{thematrix} for the matrix $L$, we notice that the entries of $L$ are rational functions of degrees $\OOO(d)$ in $c_1,c_2$; similarly for the entries of the matrix $Q=L\cdot M^{-1}$. From\eqref{elbe}, we observe that $\bfb$ is a rational function of $c_1,c_2$ with degree $\OOO(d)$. The operations involved in this step are essentially multiplication and addition of bivariate polynomials of degrees $\OOO(d)$, which can be done in $\tOOO(d^2)$ time \cite{Pan}.

We address now Step 3, i.e. the derivation of the polynomial system in $c_1,c_2$ from equation \eqref{reference-eq}. Let us denote 
\[\bfx(t,s)=\sum_{\begin{array}{c} i,j=0 \\ i+j\leq d\end{array}}^d \bar{\alpha}_{i,j}\cdot t^i s^j.\]We want to compute 
\[\bfx\left(\varphi(t,s)\right)=\sum_{\begin{array}{c} i,j=0 \\ i+j\leq d\end{array}}^d \bar{\alpha}_{i,j}\cdot (at+bs+c_1)^i (ct+ds+c_2)^j,\]
where $a,b,c,d$ are rational functions of $c_1,c_2$ of degrees $\OOO(d)$. Each $(at+bs+c_1)^i$ or $(ct+ds+c_2)^j$ can be computed in $\tOOO(d^4)$ time by using binary exponentiation and FFT-based multiplication \cite[\S 8.2]{VonZurGathen.Gerhard}, and yields a polynomial where the coefficients are rational functions of $c_1,c_2$ of degrees bounded by $\OOO(d^2)$. The product $(at+bs+c_1)^i\cdot(ct+ds+c_2)^j$ is computed after $\tOOO(d^4)$ operations. Since each component of $\bfx(t,s)$ has at most $\OOO(d^2)$ terms (as a polynomial in $t,s$), we have to repeat this process $\OOO(d^2)$ times, therefore yielding a total complexity of $\tOOO(d^6)$ for this step. The bivariate polynomial system in $c_1,c_2$ derived this way has degree $\OOO(d^2)$ and consists of $\OOO(d^2)$ equations. 

Finally we consider Step 4, i.e. solving the system. The complexity of determining the real solutions of a (possibly overdetermined, and non necessarily zero-dimensional) polynomial system of $k$ equations in $n$ variables, with degrees bounded by $D$, is $\OOO\left((kD)^{n^2}\right)$ \cite{grigoriev}. Since in our case $k=\OOO(d^2)$, $n=2$, $D=\OOO(d^2)$, we get a complexity of $\OOO(d^{16})$ for this step.

Step 5 does not add any complexity to the previous steps. Therefore, we get an overall complexity of $\OOO(d^{16})$. This complexity is dominated by that of Step 4, which is certainly the bottleneck of the algorithm.

\subsection{Performance} We have implemented the algorithm {\tt SymSurf} in the computer algebra system Maple 17. The examples have been run on an intel Core i7, revving up to 2.90 GHz, with 8 Gb RAM. We list the features of some of these examples in Table 1. More precisely, in each case we provide the bidegree $(d_1,d_2)$ of the parametrization, the absolute value of the maximum coefficient of the parametrization, the timing, and the involutions found. 

\begin{table}
\begin{tabular}{lcccc}
\toprule
Surface & Bidegree & Coeffs. & Timing & Obs. \\
\midrule
Elliptic paraboloid & (2,2) & 18 & 0.374 & Planar, axial. \\ 
Hyperbolic paraboloid & (2,2) & 16 & 0.234 & Planar, axial. \\ 
Circular paraboloid & (2,2) & 1 & 0.078 & Revol. surf.\\  
Enneper surface & (3,3) & 3 & 0.141 & Planar, axial. \\ 
Example 8 & (3,3) & 3 & 0.187 & Planar, axial. \\ 
Example 9 & (4,4) & 6 & 0.827 & Planar, axial. \\ 
Example 10 & (5,5) & 10 & 5.912 & Planar, axial. \\ 
Example 2 & (6,6) & 20 & 3.073 & Planar, axial. \\ 
Example 1 & (7,7) & 35 & 8.565 & Central. \\ 
Example 11 & (8,8) & 70 & 82.213 & Planar, axial. \\ 
Revol.2 & (8,8) & 6 & 0.640 & Revol. surf. \\ 
Example 6 & (9,11) & 924 & 196.25 & Central.
\end{tabular}
\caption{Average CPU time (seconds) for involutions of several polynomially parametrized surfaces.}\label{tab:involutions}
\end{table}

The table shows a good performance for surfaces of moderate bidegrees. The bottleneck of the algorithm, as shown in the complexity section, is the isolation of the real roots of the bivariate systems corresponding to the cases in Lemma \ref{config-complete}; this explains the explosion in the time as the bidegree grows.  

\section{Conclusions and Future Work.}\label{conclu}

We have presented a new algorithm to compute the involutions of an algebraic surface admitting a polynomial parametrization, under the hypotheses that the parametrization defining the surface is proper and normal. Our method stems from the fact that any involution $f$ of the surface comes from an involution $\varphi(t,s)$ in the parameter space, which is proven to be a linear mapping. All these ideas are used to write both $f$ and $\varphi$ in terms of just two parameters, so that the problem of computing the involutions leaving the surface invariant is reduced to checking whether or not certain bivariate systems admit a real solution. 

It is natural to wonder if the method is generalizable to the case of rational parametrizations. Certainly, the idea of reducing the problem to computations in the parameter space is still valid in that case, because in the rational, not necessarily polynomial, case, and under similar hypotheses, we still have a, in general, rational function $\varphi(t,s)$ making a diagram like \eqref{eq:fundamentaldiagram} conmutative. However, two difficulties arise here. First, the mapping $\varphi$ is no longer linear. Second, even if the general form of $\varphi$ could be found, quite likely the number of parameters of $\varphi$ would be higher, and therefore the current method would turn somehow cumbersome. So we believe that extra ingredients should be combined to solve the rational, not necessarily polynomial, case.

One might also wonder how to compute the symmetries of an algebraic surface implicitly defined. To our knowledge there is no known algorithm to solve this question. This is certainly a nice and challenging problem which we would like to address in the future.

\end{document}